\documentclass[12pt]{amsart}
\usepackage[active]{srcltx}
\usepackage{a4wide}
\usepackage{amsthm,amsfonts,amsmath,mathrsfs,amssymb}
\usepackage{dsfont}

\newcommand{\D}{\mathbb{D}}

\newcommand{\RR}{\mathbb{R}}

\newcommand{\Z}{\mathbb{Z}}
\newcommand{\T}{\mathbb{T}}

\newcommand{\Ht}{H^2}

\newcommand{\Real}{\operatorname{Re}}

\newcommand{\vspan}{\operatorname{span}}
\newcommand{\Us}{\mathcal{U}_{s}}
\newcommand{\Uc}{\mathcal{U}_{c}}
\newcommand{\Ug}{\mathcal{U}}
\newcommand{\Cp}{C_{\varphi}}
\newtheorem{theorem}{Theorem}[section]

\newtheorem{lemma}{Lemma}[section]
\newtheorem{corollary}{Corollary}[section]

\begin{document}

\title[Approximation numbers of composition operators]{Decay rates for approximation numbers of composition operators}


\author[H. Queff\'{e}lec]{Herv\'{e} Queff\'{e}lec}
\address{Universit\'{e} Lille Nord de France, USTL, Laboratoire Paul Painlev\'{e} U.~M.~R. CNRS 8524, F--59 655 Villeneuve d'Ascq
Cedex, France}
\email{herve.queffelec@univ-lille1.fr}


\author[K. Seip]{Kristian Seip}
\address{Department of Mathematical Sciences, Norwegian University of Science and Technology,
NO-7491 Trondheim, Norway} \email{seip@math.ntnu.no}
\thanks{The second author is supported by the Research Council of Norway
grant 160192/V30. This paper was written while the authors participated in the research program \emph{Operator Related Function Theory and Time-Frequency Analysis} at the Centre for Advanced Study at the Norwegian Academy of Science and Letters in Oslo during 2012--2013.}



\begin{abstract}
A general method for estimating the approximation numbers of composition operators on $\Ht$, using finite-dimensional model subspaces, is studied and applied in the case when the symbol of the operator maps the unit disc to a domain whose boundary meets the unit circle at just one point. The exact rate of decay of the approximation numbers is identified when this map is sufficiently smooth at the point of tangency; it follows that a composition operator with any prescribed slow decay of its approximation numbers can be explicitly constructed. Similarly, an asymptotic expression for the approximation numbers is found when the mapping has a sharp cusp at the distinguished boundary point. Precise asymptotic estimates in the intermediate cases, including that of maps with a corner at the distinguished boundary point, are also established.   
\end{abstract}

\maketitle

\section{Introduction}
This paper studies a general method for estimating the approximation numbers of compact composition operators $\Cp$ on the Hardy space $\Ht$, where as usual $\Cp$ is defined by the relation $\Cp f=f\circ \varphi$ and the symbol of the operator $\varphi$ is an analytic function mapping the unit disc $\D$ into itself. Our main application will be to identify as precisely as possible the rate of decay for the approximation numbers $a_n(\Cp)$ when $\varphi(\D)$ touches the unit circle $\T$ at just one point. We will simplify matters by considering symbols of the form $\varphi=\exp(-u-i \tilde{u})$, where the real valued function $u$ is in $C(\T)$, satisfies $u(z)=u(\overline{z})$, and is smooth except possibly at $z=1$. Moreover, we will assume that the even function $U(t) := u(e^{it})$ is increasing on $[0,\pi]$ and that $U(0)=0$. The associated composition operator will be compact on $H^2$ if and only if  $U(x)/x^2$ is non-integrable or, in other words, if the function
\[ h_U(t):= \int_{t}^{\pi} \frac{U(x)}{x^2} dx \]
is unbounded for $0<t\le \pi$. The functions $U$ that satisfy these conditions, including $h_U(t)\to\infty$ when $t\to 0^+$, will be said to belong to the class $\Ug$.

We obtain the most precise results if either $h_U$ grows slowly or $U$ tends slowly to $0$ at $0$. To deal with the latter situation, we introduce two other auxiliary functions, defined in the following way. We write
\[ U(t)=e^{-\eta_U(| \log t|)} \]
whenever $0<t\le 1$ and $U(t)\le e^{-1}$, and define $\omega_U$  by the implicit relation
\[ \eta_U\left(x/\omega_U(x)\right)=\omega_U(x) \]
for those $x\ge 0$ such that $\eta_U(x)\ge 1$. The monotonicity of $\eta_U$ implies that $\omega_U$ is an increasing function. 

We use the notation $a_n(T)$ for the $n$th approximation number of a bounded operator on a Hilbert space. This is the distance in the operator norm from $T$ to the operators of rank $<n$.  With every function $U$ in $\Ug$ we associate the symbol $\varphi_U:=\exp(-u-i v)$, where $u(e^{it})=U(t)$ and $v$ is the harmonic conjugate of $u$.

The two extreme cases alluded to above are covered by the following theorem. 
\begin{theorem}\label{extreme}
Suppose that $U$ belongs to $\Ug$. 
\begin{itemize}
\item[(a)] If $t U'(t)/U(t)\le 1+c/|\log t|$ and $U(t)/(t h_U(t))\le C(|\log t| \log |\log t|)^{-1}$ for $c>1$, $C>0$, and $t>0$ sufficiently small, then \[
a_n(C_{\varphi_U})=e^{O(1)}/\sqrt{h_U(e^{-\sqrt{n}})}\] 
when $n\to\infty$.
\item[(b)] If $\eta_U'(x)/\eta_U(x)=o(1/x)$ when $x\to \infty$, then \[a_n(C_{\varphi_U})=\exp(-(\pi^2/2+o(1))n/\omega_U(n))\] when $n\to\infty$.
\end{itemize}
\end{theorem}
Here part (a) corresponds to smooth tangency at $1$, while part (b) deals with the case of a sharp cusp at $1$. In (a), one may think of 
$h_U(1/x)$ as a power of $\log \log x$ or any function growing more slowly than this; in (b),  a typical case is $U(t)=1/|\log t|$ for small $t>0$.  

Part (a) of Theorem~\ref{extreme} has the following corollary.

\begin{corollary}\label{slowdecay}
Let $g$ be a function on $\RR^+$ such that $g(x)\searrow 0$ when $x\to \infty$ and $g(x^2)/g(x)$ is bounded below. Then there exists a compact composition operator $\Cp$ on $\Ht$ with the two-sided estimate
\[ a_n(\Cp)=e^{O(1)} g(n)\]
when $n\to\infty$.
\end{corollary}
Observing that the boundedness condition on $g(x^2)/g(x)$ allows us to assume that $|g'(x)|/g(x)\le C/(x \log x)$ and $g''(x)/|g'(x)|\le C/x$ for large $x$, we obtain the corollary from part (a) of Theorem~\ref{extreme} by simply setting 
\[ h_U(t)=[g((\log t)^2)]^{-2} \]
for small $t$. 

Corollary~\ref{slowdecay} says that we may prescribe any slow rate of decay (a negative power of $\log n$ or slower) and find a function $U$ in $\Ug$ such that the approximation numbers $a_n(C_{\varphi_U})$ descend accordingly. This result can be seen to originate in a question raised by Sarason in 1988 \cite{COW}: Do there exist compact composition operators on $\Ht$ that do not not belong to any Schatten class $S_p$? Carroll and Cowen \cite{CARCOW} gave an affirmative answer three years later, and three alternate proofs, relying on criteria for membership in the Schatten class $S_p$ due to Luecking and K. Zhu \cite{LZ}, were given in respectively \cite{Z}, \cite{JON}, \cite{QRLL}. A fourth proof  was given in \cite[Remark 2.7]{AKE}, which avoids the use of the Luecking--Zhu criterion. It was finally shown in \cite{LIQUEROD} that Carroll and Cowen's method allows the construction of compact composition operators with approximation numbers descending arbitrarily slowly. 

In spite of its flexibility, a drawback of the  Carroll--Cowen construction  
is that it does not give any clue about the  behavior of the curve $t\mapsto \varphi(e^{it})$, and hence one can not use the MacCluer criterion \cite{MAC} to give an upper bound for the approximation numbers. Indeed, compactness is ensured by the Julia-Carath\'eodory theorem 
which gives no extra quantitative information, and in particular no upper bound. The desire to obtain a more explicit example in which one has control of the curve $t\mapsto \varphi(e^{it})$, through a systematic use of conjugate functions, was an initial motivation for the present work. 

We will deal with the whole range of possible behaviors that may occur ``between" the two extremes described by Theorem~\ref{extreme}, but then the results are a little less precise. The case corresponding to maps with a corner at  one point is illustrative, as it sits on the edge between the two different kinds of asymptotics exhibited in Theorem~\ref{extreme}:

\begin{theorem}\label{sousex} Suppose that $\varphi(z)=(1+(1-z)^{\alpha})^{-1}$ for some $0<\alpha<1$.
Then
\[ \exp\left(-\pi(1-\alpha)\sqrt{2n/\alpha}\right)\ll a_{n}(C_{\varphi})\ll \exp\left(-\pi(1-\alpha) \sqrt{n/(2\alpha)}\right). \]
\end{theorem}
Here the notation $f(n)\ll g(n)$ (or equivalently $g(n)\gg f(n)$) means that $f(n)\le C g(n)$ for all $n$. We observe that the symbols in Theorem~\ref{sousex} can be written as $\varphi=\varphi_U$ with $U(t)=\Real \log(1+(1-e^{it})^{\alpha})$ and that the corresponding auxiliary function $h_U(t)$ behaves, up to a constant factor, as $t^{\alpha-1}$ for small $t$. If we extrapolate from part (a) of Theorem~\ref{extreme}, then this suggests why we have the factor $1-\alpha$ in the exponential in Theorem~\ref{sousex}. On the other hand, we observe that now
\[ \omega_U(x)=(1+o(1))\sqrt{\alpha x} \]
when $x\to \infty$, which shows that Theorem~\ref{sousex} also agrees with part (b) of Theorem~\ref{extreme}. In Section 4,  we will present a general theorem displaying the two types of asymptotics, based on a division of $\Ug$ into functions corresponding, roughly speaking, to smooth maps and maps with a cusp; the functions $U$ covered by Theorem~\ref{sousex} lie in the interface between these two subclasses of $\Ug$.  

Note that in Theorem~\ref{sousex}, we could just as well have considered so-called lens maps which have corners at two opposite points on $\T$. See  \cite{LLQR} and \cite{LIQUEROD}  for similar but somewhat less precise estimates for such maps. Indeed, our method allows us to consider symbols $\varphi$ such that $\varphi(\D)$ touches the unit circle at a finite number of points (that case was also  considered in \cite{LIQURO})   and $\varphi$ behaves similarly as described above around each of these points. 

Our results rest on certain techniques for estimating approximation numbers that already appeared in the proofs of Proposition 6.3 in \cite{LIQUEROD} (lower bounds) and in Theorem 2.3 and Theorem 3.2 in \cite{LIQURO} (respectively upper and lower bounds). We wish to underline the generality of these ideas and stress that they lie entirely within the realm of finite-dimensional model subspaces, i.e., subspaces of $\Ht$ of the form
\[ K_B^2:=H^2 \ominus B H^2, \]
where $B$ is a finite Blaschke product. For the bounds from above, we use $\Cp P_B$ as the approximating operator, where $P_B$ denotes orthogonal projection from $\Ht$ onto $K^2_B$. The bounds from below rest on the fact that $\Cp^*$ maps reproducing kernels to reproducing kernels and in particular $K_B^2$ onto another model subspace when each zero of $B$ has multiplicity one. This approach leads to explicit lower and upper bounds for the approximation numbers in terms of estimates for Blaschke products, constants of interpolation, and Carleson measures; these bounds are stated in Section~3 in the form of a general theorem. Our result gives somewhat stronger general estimates than the ones that can be extracted directly from the proofs in \cite{LIQUEROD} and \cite{LIQURO}, especially because we have two-sided estimates in the case of slow decay. 

In Section~3, which presents this general method, we have also found it natural to display some generic choices for the finite Blaschke products that go into the respective bounds. This part is mainly about how one may employ the Poincar\'e metric of the disc in this context. We do not pretend that these ideas are exhaustive; more complex symbols $\varphi$ may very well require more delicate constructions.

After the generalities of Section~3, we turn in Section~4 to a general theorem for two natural subclasses of $\Ug$ exhibiting the same type of division as we already saw in Theorem~\ref{extreme}. Our three theorems (including Corollary \ref{slowdecay}) concerning the class $\Ug$ are then proved in the two remaining sections of the paper. Here we rely on the general scheme of Section~3 and a number of concrete estimates for the different kinds of harmonic conjugate functions that appear in the respective symbols $\varphi_U$. 

The idea of using finite-dimensional spaces spanned by reproducing kernels clearly makes sense in a wider context. In a closely related paper \cite{QS}, we have studied the decay rates for the approximation numbers of certain composition operators on the $H^2$ space of square summable Dirichlet series. Here the basic scheme is the same, but the technical challenges are of a quite different nature.

\section{Preliminaries}\label{background}

\subsection{Hyperbolic length and the hyperbolic metric}
We use the convention that the hyperbolic length of a curve $\Gamma$ in the Poincar\'{e} metric of $\D$ is given by the integral 
\[ \ell_P(\Gamma)=2\int_{\Gamma} \frac{|dz|}{1-|z|^2}.\]
The geodesics of the hyperbolic metric are straight lines through $0$ or circles having perpendicular intersections with the unit circle $\T$. The hyperbolic distance between two points $z$ and $w$ in $\D$ is denoted by $d(z,w)$. As in \cite[p. 4]{Garnett-livre}, we have chosen to normalize the hyperbolic metric so that
\begin{equation}\label{eta} \varrho(z,w):=\left|\frac{z-w}{1-\overline{w}z}\right|=\frac{1-e^{-d(z,w)}}{1+ e^{-d(z,w)}}; \end{equation}
here $\varrho(z,w)$ is the pseudohyperbolic distance between $z$ and $w$. We will use the notation
\[ D(z,R):=\{w\in \D: \  d(z,w)<R\}.\]
    
\subsection{Carleson measures and interpolating sequences} 

We will use two well known results about Carleson measures and interpolating sequences. We state them here in the form of two theorems; we  have added two lemmas which give dual statements in terms of reproducing kernels.

We begin with a classical result of Carleson. A nonnegative Borel measure $\mu$ on $\overline{\D}$ is a Carleson measure for $H^2$ if there exists a positive constant $K$ such that 
\[ \int_{\Omega} |f(z)|^2 d\mu(z)\le K \| f\|_{H^2}^2 \]
for every $f$ in $H^2$. The smallest possible $K$ in this inequality is called the ``Carleson norm'' of $\mu$. 
We denote it by 
$\|\mu\|_{\mathcal{C}}$ and put $\|\mu\|_{\mathcal C}=\infty$ if $\mu$ fails to be a Carleson measure. 

We will need Carleson's characterization of Carleson measures for $\Ht$ \cite{Cac}. A set of the form 
\[Q(r_0,t_0):=\{z=r e^{i t} \in \overline{\D}: \ r\ge r_0, \ |t-t_0|\le (1-r_0)\pi\},\]
where $0<r_0<1$, is   declared to be a Carleson square in $\D$, and we set $\ell(Q(r_0,t_0)):=1-r_0$.

\begin{theorem}[Carleson's theorem]\label{carleson} Let $\mu$ be a nonnegative Borel measure on $\overline{\D}$ and let $\|\mu\|_{\mathcal C}$ be the Carleson norm of $\mu$ with respect to $\Ht$. 
There exists an absolute constant $C$ such that \[\| \mu \|_{\mathcal C}\le C \sup_{Q} \mu(Q)/\ell(Q),\] where the supremum is taken over all Carleson squares $Q$ in $\overline{\D}$.
\end{theorem}

A sequence  $Z=(z_j)$  in $\D$ is said to be a Carleson sequence for $\Ht$ if the measure 
\[\upsilon_{Z}:=\sum_j (1-|z_j|^2)\delta_{z_j}\] 
is a Carleson measure for $\Ht$.  We will  now give a dual statement about Carleson sequences in terms of the reproducing kernel
\[k_{w}(z):=\frac{1}{1-\overline{w}{z}}\]
for $\Ht$. The proof is straightforward and can be found in \cite{QS}.  

\begin{lemma}\label{disk}
If $Z=(z_j)$ is a Carleson sequence for $\Ht$, then 
\[ \Big\| \sum_{j} b_j k_{z_j}\Big\|_{\Ht}^2\le \|\upsilon_{Z}\|_{\mathcal C} \sum_j |b_j|^2(1-|z_j|^2)^{-1} \]
for every finite sequence of complex numbers $(b_j)$.
\end{lemma}

We will say that a sequence $Z=(z_j)$ in $\D$ is an interpolating sequence for $\Ht$ if the interpolation problem $f(z_j)=a_j$ has a solution $f$ in 
$\Ht$ whenever  the admissibility condition 
\begin{equation} \label{adm} 
\sum_j |a_j|^2 (1-|z_j|^2)<\infty\end{equation}
holds.  If $Z$ is an interpolating sequence for $\Ht$, then the open mapping theorem shows that there is a constant $C$ such that we can solve $f(z_j)=a_j$ with the estimate 
\[ \| f\|_{\Ht}\le C \left(\sum_j |a_j|^2 (1-|z_j|^2)\right)^{1/2} \]
whenever \eqref{adm} holds. The smallest $C$ with this property is denoted by $M(Z)$, and we call it the constant of interpolation. 

The sequence $Z$ is said to be separated if  $\inf_{j\neq k}\varrho(z_j,z_k)>0$. A more severe notion of separation can be defined by means of the  quantity 
\[ \delta(Z):=\inf_{j}\prod_{k: k\neq j} \varrho(z_j,z_k).\]
The following theorem was obtained from Carleson's work \cite{Cai} by Shapiro and Shields \cite{SS}.
\begin{theorem}[Shapiro--Shields's theorem]\label{CSS}
A sequence $Z$ of distinct points in $\D$ is an interpolating sequence for $\Ht$ if and only if $\delta(Z)>0$.
Moreover,
\[ M(Z) \le \|\mu_Z\|_{\mathcal C}^{1/2}/\delta(Z). 
\] 
\end{theorem}  
The estimate given above for $M(Z)$ is obtained from a duality argument that can be found in \cite[p. 227]{ScSe}.
We will need the following dual version of Theorem~\ref{CSS}.

\begin{lemma}\label{interdual}
If $Z=(z_j)$ is an interpolating sequence for $\Ht$, then
\[ \Big\| \sum_{j} b_j k_{z_j}\Big\|_{\Ht}^2\ge [M(Z)]^{-2} \sum_j |b_j|^2(1-|z_j|^2)^{-1} \]
for every finite sequence of complex numbers $(b_j)$.
\end{lemma}
This reformulation is classical and seems to have been observed first by Boas \cite{Bo}. 

\subsection{Bernstein numbers} 

We will make use of the following general characterization of $n$th approximation numbers.

\begin{lemma}\label{bernstein} Let $T$ be a bounded operator on a Hilbert space $H$. Then
\begin{equation}\label{deuz} a_n(T)=\sup_{\dim E=n}\Big[ \inf_{x\in E, \|x\|=1}\Vert Tx\Vert\Big].\end{equation}
\end{lemma} 
The proof is elementary and can be found in \cite{PIE}.  The number defined by the right-hand side of \eqref{deuz} is called the $n$th Bernstein number of $T$. 

One may use Lemma~\ref{bernstein} to establish lower bounds for $a_n(T)$. The efficiency of this method depends on whether a good choice of $E$ can be made. In our case, when $T=C_{\varphi}^{*}$, we will  take advantage of the relation 
 \begin{equation}\label{identity}C_{\varphi}^{*}(k_a)=k_{\varphi(a)}\end{equation}
which holds for every point $a$ in $\D$. We will choose $E$ as a linear span of a suitable finite sequence of reproducing kernels. 
\section{The general method}

\subsection{A general theorem}
 To state our result for general composition operators, we introduce the following standard pullback measure, which may be associated with any composition operator $\Cp$. Let $\sigma$ denote normalized Lebesgue measure on the unit circle $\T$ and set $\mu_\varphi:=\varphi(\sigma)$. For $0<r<1$, $\mu_{\varphi,r}$ is the nonnegative Borel measure on $\overline{\D}$ carried by the annulus $\{w \ :\  r<\vert w\vert\le1\}$, defined by the requirement that
\begin{equation} \label{defcarl}\mu_{\varphi,r}(E):=\sigma(\{z: |\varphi(z)|> r \ \text{and} \ \varphi(z)\in E\})=\sigma\left(\varphi^{-1}(E)\cap \{|\varphi(z)|>r\}\right).\end{equation}
Equivalently, we may write $\mu_{\varphi,r}(E)=\mu_{\varphi}\left(E\cap (\D \setminus r\overline{\D})\right)$.
A finite Blaschke product $B$ of degree $n-1$ is a function of the form
\[ B(z)=z^{m}\prod_{j=1}^{n-m-1}\frac{z_j-z}{1-\overline{z_j} z}, \] 
where $0\le m \le n-1$ and $(z_j)$ is a sequence of not necessarily distinct points in $\D\setminus\{0\}$.
\begin{theorem}\label{abovebelow}
Let $\Cp$ be a composition operator on $\Ht$.
\begin{itemize}
\item[(a)] Let $B$ be an arbitrary Blaschke product of degree $n-1$ and $0<r<1$. Then 
\[  a_{n}(\Cp)\le\left(\sup_{z\in\T: |\varphi(z)|\le r} |B(\varphi(z))|^2 \|\Cp\|^2 +\|\mu_{\varphi,r}\|_{\mathcal C}\right)^{1/2}.
 \]
\item[(b)]
Let $Z=(z_j)$  be a finite sequence  consisting of $n$ distinct points in $\D$, as well as the sequence $\varphi(Z)$ of their images. Then
\[ a_n(\Cp)\ge [ M(\varphi(Z)) ]^{-1} \|\upsilon_{Z}\|_{\mathcal C}^{-1/2} \inf_{1\le j\le n} \left(\frac{1-|z_j|^2}{1-|\varphi(z_j)|^2}\right)^{1/2}. \]
\end{itemize}
\end{theorem}

\begin{proof}
We begin with part (a). Starting from the definition of the $n$th approximation number $a_n(C_{\varphi})$, we see that if $R_{n-1}$ is an arbitrary operator of rank $n-1$, then
\begin{equation}\label{fromdef} a_{n}(C_{\varphi})\le \| \Cp - R_{n-1}\|. \end{equation}
We now use the following rank $n-1$ operator. Let $B$ be an arbitrary finite Blaschke product of degree $n-1$. With $B$ we associate the model subspace $K_{B}^2$ which we defined in the introduction,
and we let $P_B$ denote the orthogonal projection from $\Ht$ onto $K^2_B$. We then set $R_{n-1}:=\Cp\circ P_B$ and note that both $P_B$ and $R_{n-1}$ 
are operators of rank $n-1$. If $f$ is an arbitrary function in $\Ht$ with norm one,  then we have 
\begin{equation} \label{aboveI}\Vert (C_\varphi -R_{n-1})f\Vert_{\Ht}^{2}=\int_{\T} \vert G(z)\vert^2 d\sigma(z),\end{equation}
where 
$$G(z):=f(\varphi(z))-P_{B}f(\varphi(z))=B(\varphi(z)) F(\varphi(z))$$
and $\Vert F\Vert_{\Ht}\leq \Vert f\Vert_{\Ht}= 1$. Here $G(z)$ is the radial limit of $(\Cp-R_{n-1})f$ which exists for almost every point $z$ in $\T$. It follows from \eqref{aboveI} that, for any $0<r<1$, we have
\begin{equation}\label{basicest}
\Vert (C_\varphi -R_{n-1})f\Vert_{\Ht}^{2}\leq\sup_{z\in \T: |\varphi(z)|\le r} |B(\varphi(z)|^2 \|\Cp\|^2+ \int_{z\in \T: |\varphi(z)|> r} \vert F(\varphi(z))\vert^2 d\sigma(z).
\end{equation}
Returning to \eqref{fromdef} and using the definition of $\mu_{\varphi,r}$, we obtain part (a) of the theorem.

We now turn to part (b). For an arbitrary finite sequence $Z=(z_1,\ldots,z_n)$  of $n$ distinct points in $\D$, we set
\[ E(Z):=\vspan\{k_{z_1},\ldots, k_{z_n}\}\]
or in other words $E(Z)=K^2_B$, where $B$ is the finite Blaschke product with zeros $z_1,\ldots,z_n$.
We assume that both $Z$ and $\varphi(Z)$ consist of $n$ distinct points in $\D$. By \eqref{identity}, $\Cp^*$ is a bijection from $E(Z)$ onto $E(\varphi(Z))$. According to Lemma~\ref{bernstein}, we have
\[ a_n(\Cp) \ge   \inf_{f\in E(Z), \|f\|=1}\Vert \Cp^* f\Vert.\]
Using the two bounds (see also \cite{LIQURO}, Lemma 3.3)
\[ \Big\| \sum_{j} b_j k_{z_j}\Big\|_{\Ht}^2\le \|\upsilon_{Z}\|_{\mathcal C} \sum_j |b_j|^2(1-|z_j|^2)^{-1} \]
and 
\[ \Big\| \sum_{j} b_j k_{\varphi(z_j)}\Big\|_{\Ht}^2\ge [M(\varphi(Z))]^{-2} \sum_j |b_j|^2(1-|\varphi(z_j)|^2)^{-1} \]
from respectively Lemma~\ref{disk} and Lemma~\ref{interdual}, we therefore arrive at part (b).
\end{proof}

Thus our method for finding an upper bound consists in finding suitable Blaschke products $B$, estimating the Carleson norm $\|\mu_{\varphi,r}\|_{\mathcal C}$, and in combining our choices for $B$ and $r$ in order to minimize the right-hand side in part (a) of Theorem~\ref{abovebelow}. To find a lower bound for $a_n(\Cp)$, we would like to have a sequence $Z$ such that each of the three factors on the right-hand side in part (b) becomes large. Clearly, the third factor becomes large if the distances $1-|z_j|$ are large, but then the two first factors will be small, and again the matter is to find a reasonable tradeoff. 

It should be noted that part (b) of Theorem~\ref{abovebelow} can be refined when $\varphi$ fails to be injective; then we may replace each reproducing kernel $k_{z_j}$ by a suitable linear combination of reproducing kernels $k_{z_{j,\ell}}$ such that $\varphi(z_{j,\ell})$ has the same value for all $\ell$. This idea is elaborated in \cite[Theorem 4.1]{QS} to obtain sharp results in a multi-dimensional context. To illustrate the potential improvement in the present situation, we introduce the function
\begin{equation}\label{nevandef} N_{\varphi}^{*}(w):=\sum_{z\in\varphi^{-1}(w)} (1-|z|^{2})\end{equation}
which is a version of the classical Nevanlinna counting function, adapted to our setting. Then, given a sequence $W=(w_j)_{j=1}^n$, we may set
\[ g_j(w)=[N_{\varphi}^{*}(w)]^{-1} \sum_{z\in \varphi^{-1}(w_j)} (1-|z|^2) k_z(w) \]
and use, instead of $E(Z)$ from the preceding proof, the $n$-dimensional space
\[ E=\vspan \{g_1,\ldots,g_n\}. \]
With this replacement in our proof of part (b) of Theorem~\ref{abovebelow} we then obtain 
\begin{equation}\label{nevanl}  
a_n(\Cp)\ge [ M(W) ]^{-1} \|\upsilon_{\varphi^{-1}(W)}\|_{\mathcal C}^{-1/2} \inf_{1\le j\le n} \left(\frac{N^*_\varphi(w_j)}{1-|w_j|^2}\right)^{1/2}.
\end{equation}
This inequality has the advantage that the last factor on the right-hand side is the quantity used in Shapiro's formula for the essential norm of a composition operator \cite{Shap}. But to apply \eqref{nevanl}, one would need to find a way to control the Carleson norms $\|\upsilon_{\varphi^{-1}(W)}\|_{\mathcal C}$, which appears to be a nontrivial problem. Note, however, that it would suffice to pick a subset of $\varphi^{-1}(w_j)$
whose contribution to the sum in \eqref{nevandef} is bounded below by $N^{*}(w_j)$ times a constant independent of $j$.

Returning to part (a) of Theorem~\ref{abovebelow}, we note that the simplest possible choice we can make for $B$ is to set $B(z)=z^{n-1}$. This gives
\[ \left[a_{n}(\Cp)\right]^2\le\inf_{0<r<1} (r^{n-1} \|\Cp\|^2 +\|\mu_{\varphi,r}\|_{\mathcal C}), \]
which in particular yields the well-known fact that $\Cp$ is compact if  $\|\mu_{\varphi,r}\|_{\mathcal C}\to 0$ when $r\to 1$, which is easily seen to be equivalent to the MacCluer condition that the pullback measure $\sigma\circ\varphi^{-1}$ is a vanishing Carleson measure. If no additional information is available about $\Cp$, or it is known that $\varphi(\T)$ has in some sense a bad localization in $\D$, then it is reasonable to set $B(z)=z^{n-1}$. But in our situation, we will make two different choices that will give much better estimates.

\subsection{First example of choice for a Blaschke product  $B$ in Theorem~\ref{abovebelow}} \label{First} We will now describe the choice that we will later make when $\varphi_U(\T)$ is smooth. The approach is completely general and should therefore be viewed as part of the general method.

Let us for convenience set $E_r=\{z\in \T : |\varphi(z)|\le r \}$ and $\Omega_r:=\varphi(E_r)$. We assume that $\varphi$ is differentiable on $E_r$ and that the curve $\Omega_r$ is connected for every $0<r<1$. This means in particular that the curve has two end-points $z_0$ and $z_1$. 
We assume that $n$ grows with $r$ so that $\ell_P(\Omega_r)/n =o(1)$ when $r\to 1$. Choose accordingly, for every $r$, an integer $m$ such that $m=o(n)$ and $m \ell_P(\Omega_r)/n \to \infty$ when $r\to 1$. This is clearly possible since $\ell_P(\Omega_r)\to \infty$ when $r\to 1$. We now choose $n-2m-2$ points $z_2, \ldots, z_{n-2m-1}$ along the curve $\Omega_r$ such that the hyperbolic length of the curve between any two points $z_j$ and $z_{j+1}$ is $\ell_P(\Omega_r)/(n-2m-2)$, where we for convenience have declared that $z_{n-2m-1}:=z_0$. We require $B$ to have double zeros at $z_1,\ldots,z_m$ and $z_{n-2m-1},\ldots,z_{n-3m}$ and simple zeros at $z_{m+1},\ldots,z_{n-3m-1}$. We observe that the degree of $B$ is $ 2m+2m+n-4m-1=n-1$.

\begin{lemma}\label{variant1}  Let $\varphi$ and $B$ be as above. Then
\[ \sup_{z\in E_r}|B(\varphi(z))| \le \exp\left(- (\pi^2/2+o(1))n/ \ell_P(\Omega_r)\right)
 \]
when $r\to 1$.
\end{lemma} 
\begin{proof}
Set $\xi:=\ell_P(\Omega_r)/(n-2m-2)$, and pick an arbitrary point $z$ in $\Omega_r$. 
We use the fact that
\begin{equation}\label{connect} \varrho(w,\tilde{w})= \frac{1-e^{-d(w,\tilde{w})}}{1+e^{-d(w,\tilde{w})}},\end{equation}
the construction of $B$, and the triangle inequality for the hyperbolic metric to deduce that
\[ |B(\varphi(z))| \le \prod_{ j=1}^{m}  \left(\frac{1-e^{-\xi j}}{1+e^{-\xi j}}\right)^2. \]
By a Riemann sum argument, this means that 
\[ |B(\varphi(z))| \le
\exp\Big(-2\xi^{-1}e^{-\xi} \int_{e^{-\xi m}}^{e^{-\xi}}\frac{1}{x} \log\frac{1+x}{1-x} dx\Big).\]
Since
\[  \int_{0}^{1}\frac{1}{x} \log\frac{1+x}{1-x} dx= \sum_{j=0}^{\infty} \frac{2}{(2j+1)^2}= \pi^2/4\]
and we have that $\xi\to 0$, $m\xi\to\infty$, and $m=o(n)$ when $r\to 1$, the result follows.
\end{proof}

\subsection{Second example of choice of a Blaschke product  $B$ in Theorem~\ref{abovebelow}} \label{Second} When we have a cusp, it seems more natural to place the zeros of $B$ on a radius. We will now assume that 
\begin{equation}\label{thin} \sup_{z\in \varphi(\T)} d(z, [0,1)) < \infty. \end{equation}

We retain the notation from the preceding subsection and assume again that $n$ grows with $r$ such that $\ell_P([0,r])/n =o(1)$ when $r\to 1$. Choose accordingly, for every $r$, an integer $m$ such that $m=o(n)$ and $m \ell_P([0,r])/n \to \infty$ when $r\to 1$. We now choose $n-3m$ points on the segment $[0,1]$  such that $0=z_0<\cdots < z_{n-3m-1}$ and the hyperbolic distance between any two points $z_j$ and $z_{j+1}$ is $\ell_P([0,r])/(n-3m-1)$. We require $B$ to have a zero of order $m$ at $0$, double zeros at $z_1,\ldots,z_m$ and $z_{n-3m-1},\ldots,z_{n-4m}$ and simple zeros at $z_{m+1},\ldots,z_{n-4m-1}$.  The degree of $B$ is then $ 5m+(n-5m-1)=n-1$. We set 
\[ \rho:=\exp(-\pi^2 n/(4 m \ell_P([0,r])))\]
and note that, by our assumption on $m$, $\rho\to 1$ when $r\to 1$.
Finally, we set
\[ \lambda:=\inf_{z\in \Omega_r: |\varphi(z)|\ge \rho} \exp(-d(\varphi(z),(z_j))\]
and introduce the constant
\begin{equation}\label{alpha} \beta:=4\sum_{j=0}^\infty \frac{\lambda^{2j+1}}{(2j+1)^2}.\end{equation}
When we have a cusp at $1$, we will have that $\lambda\to 1$ and hence $\beta\to \pi^2/2$ when $\rho\to 1$.

\begin{lemma}\label{variant2} Let $B$ be as above, and assume that \eqref{thin} holds. Then
\[ \sup_{z\in \Omega_r} |B(\varphi(z))| \le \exp\left(-(\beta+o(1))\, n/\ell_P([0,r))\right), \]
when $r \to 1$. 
\end{lemma} 
\begin{proof}
We argue similarly as in the preceding case. Hence we set $\xi:=\ell_P([0,r])/(n-3m-1)$ and consider an arbitrary  $z$ in $\Omega_r$. Set $\lambda:=\exp(-d(\varphi(z),(z_j)))$.  Then, using again \eqref{connect} and the triangle inequality for the hyperbolic metric, we get
\[ |B(\varphi(z))|
\le |\varphi(z)|^m \prod_{j=1}^m \left(\frac{1-\lambda e^{-\xi j}}{1+\lambda e^{-\xi j}}\right)^2.\]
If $|\varphi(z)|\le \rho$, then the first factor on the right-hand side gives the desired estimate. If, on the other hand, $|\varphi(z)|> \rho$, then we get similarly
\[ |B(\varphi(z))| \le
\exp\Big(-2\xi^{-1}e^{-\xi} \int_{e^{-\xi m}}^{e^{-\xi}}\frac{1}{x} \log\frac{1+\lambda x}{1-\lambda x} dx\Big)\]
and therefore
\[ |B(\varphi(z))| \le
\exp\Big(-(\beta+o(1))\, \xi^{-1}\Big) \]
when $r\to 1$.
\end{proof}

\subsection{Estimates of constants of interpolation.}
 We describe now a generic choice of sequence to get a suitable estimate from part (b) of Theorem~\ref{abovebelow}. The rectifiable curve $\Gamma: (0,2\pi) \to \D$ in the next lemma should be thought of as the image under $\varphi$ of another suitable curve in $\D$. We use the notation 
\[\Gamma_r:=\Gamma((0,2\pi))\cap\{|z|\le r\}.\]

\begin{lemma}\label{discretize}
Suppose that $\Gamma: (0,2\pi) \to \D$ is a rectifiable curve in $\D$ such that arclength $s_\Gamma$ on $\Gamma$ is a Carleson measure for $\Ht$ and $\ell_P(\Gamma)=\infty$. Then there exists a constant $C$ such that for every $n\ge \ell_P(\Gamma_r)$ and $0<r<1$ there is a sequence $Z=(z_j)$ of $n$ points on $\Gamma_{(r+1)/2}$ satisfying
\[   M(Z)\le \exp(C 
n/\ell_P(\Gamma_r)). \]
\end{lemma}

\begin{proof}  We begin by choosing inductively a sequence of points on $\Gamma_r$.  Pick a point on $\Gamma_r$, say $\gamma_1$, of minimal
modulus.  Assuming $\gamma_1,\ldots,\gamma_j$ have been chosen, we next select a point $\gamma_{j+1}$ on $\Gamma_r$ of minimal modulus in the complement of
$\bigcup_{i=1}^{j} D(\gamma_i, 1)$, as long as such a point can be found. When this process terminates, we have covered $\Gamma_r$ with
the union of say $m$ discs $\bigcup_{i=1}^{m} D(\gamma_i, 1)$. Since the assumption that $s_\Gamma$ is  a Carleson measure implies that there is a positive constant $C$ such that 
\[ \ell_{P}(\Gamma\cap D(\gamma, 1))\le C \] for any point $\gamma$ on $\Gamma$,  we will have 
$\ell_P(\Gamma_r)\le Cm$. In addition, we have by construction $d(\gamma_i,\gamma_{j})\ge 1$ whenever $i\neq j$.

We set next $\xi=(2C)^{-1}\ell_P(\Gamma_r)/n$ and $\nu=[1/(2\xi)]$. For each $1\le j\le m$, we pick $z_{j,\ell}$ on $\Gamma_{(1+r)/2}$ for $0\le \ell \le \nu
$ by the following recipe: $z_{j,0}=\gamma_{j}$ and then we pick $z_{j,\ell}$ such that $d(z_{j,\ell},z_{j,0})=\ell \xi$. This gives us a sequence of
$m(\nu+1)$  points, where
\[ m(\nu+1)\ge (C^{-1}\ell_P(\Gamma_r))\times ( Cn/\ell_P(\Gamma_r))\ge n.\]
Since $|\gamma_j|\le r$ and $d(\gamma_j,z_{j,\ell})\le 1/2$, all the points $z_{j,\ell}$ will belong to the set $\Gamma_{(r+1)/2}$.

The rest of the proof is plain.  We pick $n$ of the points $z_{j,\ell}$ and call them just $z_1,\ldots,z_n$. We  write
\[ \prod_{j\neq j'}\varrho(z_j,z_{j'})=\prod_{j: d(z_j,z_j')\le 1/2} \varrho(z_j,z_{j'})
\prod_{j: d(z_j,z_j')> 1/2} \varrho(z_j,z_{j'})=:\Pi_1\times \Pi_2.\]
By the construction of the points $z_{j,\ell}$ and the triangle inequality for the hyperbolic metric, we get
\[ \Pi_1\ge \prod_{ j=1}^{[\nu/2]}  \left(\frac{1-e^{-\xi j}}{1+e^{-\xi j}}\right)^2, \] using again \eqref{connect}.
By the inequality $(1-e^{-a})/(1+e^{-a}) \ge a^2 $ which holds for $0\le a <1/2$, this yields
\[ \Pi_1\ge \exp(-4\sum_{1\le j\le \nu} |\log(\xi j )|). \]
A Riemann sum argument then gives 
\[ \Pi_1\ge \exp(-4\xi^{-1})=\exp(-Cn/\ell_P(\Gamma_r)).\]
On the other hand, we find that
\[ \Pi_2^2\ge \exp\left(-(1+2\log 2)\sum_{j} \frac{(1-|z_j|^2)(1-|z_{j'}|^2)}{|1-\overline{z_j} z_{j'}|^2}\right).\]
Thus if we set $\mu:=\sum_j (1-|z_j|^2)\delta_{z_j}$, then we get
\[ \Pi_2^2\ge \exp\left(-(1+2\log 2)\|\mu\|_{\mathcal C}\right).\]
The result follows since $\|\mu\|_{\mathcal C}\le C
n/\ell_P(\Gamma_r)$.
\end{proof}
\section{General bounds for $a_n(C_{\varphi_U})$ when $U$ is in $\Ug$}

We begin by defining two subclasses of $\Ug$, corresponding respectively to smooth maps and maps with a cusp.

We let $\mathcal{U}_{s}$ denote the collection of those functions $U$ in $\Ug$ such that 
$U$ restricted to $(0,\pi]$ is in $C^2$, with $U'(\pi)=0$, and
\begin{equation}\label{smootht}
U(t)/t=o(\log h_U(t)) \quad \text{when} \ t\to 0 
\end{equation}
and
\begin{equation}\label{double} U'(t)/U(t)\le C/t\quad \text{and} \quad |U''(t)|/U(t)\le C/ t^2, \quad 0<t\le \pi/2,  \end{equation}
for some constant $C>0$. We have assumed less smoothness in part (a) of Theorem~\ref{extreme}, but this is just because then the relatively rapid growth of $U$ for small $t>0$ makes the assumption on the second derivative superfluous.  

We let next $\mathcal{U}_{c}$ denote the class of functions $U$ in $\Ug$ such that 
\begin{equation}\label{double2}  U'(t)/U(t)\le \alpha/t  \end{equation}
for $0<t\le \pi/2$ and some $0<\alpha<1$ depending on $U$. The latter condition means that  
$U(t)$ is bounded below by a constant times  $|t|^{\alpha}$. A calculus argument shows that if $\eta'_U/x)/\eta_U(x)=o(1/x)$ when $x\to \infty$, then 
\begin{equation}\label{strong} U'(t)/U(t)=o(1/t) \end{equation} 
when $t\to 0$,
and this means that the condition on $U$ in part (b) of Theorem~\ref{extreme} implies that $U$ is in $\Uc$.

The indices in our notations $\Us$ and $\Uc$ stand for respectively ``smooth'' and ``cusp". The functions appearing in Theorem~\ref{sousex} belong to $\Uc$ and are not in $\Us$, but, as already mentioned, one should think of them as sitting on the edge between the two classes. In fact, we will use the the same techniques in the proof of Theorem~\ref{sousex} as those we employ when dealing with functions in $\Us$. 

The function $h_U$ defined in the introduction will be essential in our study of $a_n(C_{\varphi_U})$ when $U$ is in $\Us$.
 It can be viewed as an approximate representative both for $-1/t$ times the conjugate function of $U$ and for $1/(1-r)$ times the Poisson integral of $U$, taken at a point $re^{i\theta}$, where $t=\max(|\theta|,1-r)$; the latter interpretation explains why $h_U(t)\to \infty$ when $t\to 0$ is the right condition for compactness. 
 
 We now introduce an additional auxiliary function $\gamma_U$ to be used along with $h_U$: 
\[
 \gamma_U(t)  :=  \int_{t}^{1}\frac{h_U(x)}{U(x)} dx \times (\log h_U(t)-\log h_U(1)).
\]
This function is a strictly decreasing, unbounded function on $(0,1]$, which implies that the inverse function $\gamma_U^{-1}:[0,\infty):\to (0,1]$ is well defined. We have $\gamma_U(t)\ge (\log t)^2$ because
\[ (\log t)^2= \Big(\int_{t}^{1}\frac{dx}{x}\Big)^2\le \int_t^{1}\frac{h_U(x)}{|h_U'(x)|}\frac{dx}{x^2} \times 
\int_t^{1}\frac{|h_U'(x)|}{h_U(x)}dx=\gamma_U(t)\] 
by the Cauchy--Schwarz inequality; this means in particular that $\gamma_U^{-1}(x)\ge \exp(-\sqrt{x})$.

\begin{theorem}\label{smoothcusp}
Suppose that $U$  belongs to $\Ug$. 
\begin{itemize}
\item[(a)] If $U$ is in $\Us$, then 
\[ [h_U(e^{-\sqrt{n}})]^{-\frac{1}{2}-c n/\gamma_U(e^{-\sqrt{n}})} \ll a_n(C_{\varphi_U})\ll [h_U(\gamma_U^{-1}(C n))]^{-1/2},\]
for two positive constants $c$ and $C$.
\item[(b)] If $U$ is in $\Uc$ with $U'(t)/U(t)\le \alpha/t$ for $0<t\le \pi/2$ and $0<\alpha<1$, then
\[\exp(-(\pi^2+1)/2+\varepsilon) n/\omega_U(n))\ll a_n(C_{\varphi})\ll \exp(-\kappa(\alpha) n/\omega_U(n))\]
\end{itemize}
for every $\varepsilon>0$ and a positive constant $\kappa(\alpha)$ satisfying $\kappa(\alpha)=1-\alpha$ for $\alpha$ close to $0$.
\end{theorem}

Note that $n/\gamma_{U}(e^{-\sqrt{n}})\le 1$, and hence the exponent on the left-hand side in part (a) can be replaced by $-1/2-c$. However, if 
\begin{equation}\label{expaway} \gamma_U(x)\ge c (\log x)^2 \log h_U(x) \end{equation}
for some constant $c$, then we just have
\[  [h_U(e^{-\sqrt{n}})]^{-1/2} \ll a_n(C_{\varphi_U})\ll [h_U(\gamma_U^{-1}(C n))]^{-1/2}.\] 
The bound in \eqref{expaway} holds if $U(t)/(th_U(t))\le C|\log t|^{-1}$, which for instance is satisfied when $U(t)=t|\log t|^{c}$ for small $t$ and $c\ge -1$.  Inequality \eqref{expaway} is in particular valid under the more severe hypothesis on $U(t)/(th_U(t))$ assumed in part (a) of Theorem~\ref{extreme}.

We will in the next section prove Theorem~\ref{sousex} and part (a) of respectively Theorem~\ref{smoothcusp} and Theorem~\ref{extreme} (the ``smooth" case), while part (b) of Theorem~\ref{smoothcusp} and Theorem~\ref{extreme} will be established in Section~\ref{Ucusp}. 

In what follows, we will use the notation $V(t)=v(e^{it})$, where 
$v$ is the harmonic conjugate of $u$ and $u(e^{it})=U(t)$. Both $u$ and $v$ will be viewed as functions in the closed disc $\overline{\D}$.

\section{Proofs in the ``smooth" case} \label{Usmooth}
\subsection{Estimates for Poisson integrals and partial derivatives}
In our proof of the bound from below in part (b) of Theorem~\ref{extreme}, we will need the following lemma. To simplify the notation, we let $\Us^*$ denote the union of $\Us$ and the set of those $U$ for which the hypotheses of part (a) of Theorem~\ref{extreme} are satisfied. For a smooth function $u$ of $z=re^{i\theta}$, we let  $u'_r$ and $u'_\theta$ denote the partial derivatives with respect to $r$ and $\theta$.

\begin{lemma}\label{conjugates} Suppose that $U$ belongs to $\Ug$  and let $u(re^{i\theta})$ be the positive harmonic function in $\D$ with radial limit function $u(e^{it}):=U(t)$. Then 
\begin{equation}\label{uh} u(re^{i\theta})\ge \pi^{-1}(1-r) h_U(\max(2|\theta|,1-r)).
\end{equation}
If $U$ belongs to $\Us^*$, then there exist positive constants $c$ and $C$ such that for sufficiently small $\theta>0$ and $1-r\le \theta/4$, the following holds:
\begin{eqnarray}
 \label{hv}ch_U(\theta)\ \le \  - u'_r(re^{i\theta})\ & \le &  C h_U(\theta) \\
\label{ut} |u'_{\theta}(re^{i\theta})| & \le & C h_U(\theta).
\end{eqnarray}
\end{lemma}

\begin{proof} We start from the Poisson representation 
\begin{equation}\label{poisson} u(re^{i\theta})=\frac{(1-r^2)}{2\pi} \int_{-\pi}^{\pi} \frac{U(x)}{(1-r)^2+2r(1-\cos(\theta-x))} dx, \end{equation}
which immediately gives \eqref{uh} if we restrict the integration to $|x|\ge \max(2|\theta|, (1-r))$. By the symmetry of the Poisson kernel, it also gives 
\[ -u'_r(re^{i\theta})=\frac{1}{\pi} \int_{-\pi}^{\pi} K_r(\theta,x) (U(x)-U(\theta)) dx,   \]
where
\[ K_r(\theta,x):= \frac{(1+r^2)(1-\cos(\theta-x))-(1-r)^2}{[(1-r)^2+2r(1-\cos(\theta-x))]^2}. \]
If $U$ is in $\Us$, then we use the assumption on $U'$ and the symmetry of the kernel $K_r$ to deduce that
\[ \left|\int_{|x-\theta|\le \theta/2} K_r(\theta,x) (U(x)-U(\theta)) dx
\right| \le CU(\theta)/\theta.\] 
The integration over $-\theta<x<\theta/2$ is trivially bounded by the same quantity. A computation using these facts now shows that
\[ ch_U(\theta)-CU(\theta)/\theta\le  -u_r'(re^{i\theta}) \le C'h_U(\theta)+CU(\theta)/\theta. \]
By the definition of $h_U$, we have $U(\theta)/\theta=o(h_U(\theta))$ by \eqref{smootht}, and therefore the desired result \eqref{hv} follows.
If $U$ satisfies the hypotheses in part (a) of Theorem~\ref{extreme}, then we can only use that $U(x)-U(\theta)=U'(\xi)(x-\theta)$. In this case, we obtain instead
\[   \left|\int_{-\theta}^{2\theta} K_r(\theta,x) (U(x)-U(\theta)) dx
\right| \le CU(\theta)|\log \theta| /\theta.\] 
But, by assumption, the right-hand side of the latter inequality is $o(h_U(\theta))$ when $\theta\to 0$, whence we arrive again at \eqref{hv}.

Starting once more from \eqref{poisson}, we get
\[ u'_\theta(re^{i\theta})=-\frac{(1- r^2)}{\pi} \int_{-\pi}^{\pi} \frac{\sin(\theta-x)U(x)}{[(1-r)^2+2r(1-\cos(\theta-x))]^2} dx. \]
By the symmetry of the kernel, we may replace $U(x)$ by $U(x)-U(\theta)=U'(\xi)(x-\theta)$ in the interval $|x-\theta|\le \theta/2$ and obtain a contribution of order $U(\theta)/\theta$; in the remaining part of the interval $[-\pi,\pi]$, we just do a rough estimation, and the result follows.  \end{proof}

By the Cauchy--Riemann equations in polar coordinates, \eqref{hv} implies that 
\begin{equation} \label{Vest} ch_U(t) \le |V'(t)| \le C h_U(t),\end{equation}
a fact that will be needed later.
\subsection{Estimates associated with two curves}
Set $I_{\varepsilon}:=\{e^{it}:\ |t|\le \varepsilon\}$. In order to apply Lemma~\ref{variant1}, we need to estimate $\ell_P(\varphi(\T\setminus I_{\varepsilon}))$ from above. 
\begin{lemma}\label{poincare1} Suppose that $U$ is in $\Us^*$. Then
\[ \ell_P(\varphi_U(\T\setminus I_\varepsilon))\le C \int_{\varepsilon}^\pi \frac{h_U(t)}{U(t)} dt. \]
for some positive constant $C$.  
\end{lemma}
\begin{proof}
For the map $z(t):=\exp(-U(t)-iV(t))$, we have
\[ |z'(t)| =|z(t)| \left([U'(t)]^2+[V'(t)]^2\right)^{1/2}.\]
We use the right inequality in \eqref{Vest} and the bound $U'(t)\le C h_U(t)$, which is either inferred from \eqref{smootht} and \eqref{double} or from the hypothesis of part (a) of Theorem~\ref{extreme}. We may therefore conclude that
\[ \ell_P(\varphi_U(\T\setminus I_\varepsilon)):=2\int_{\varphi_U(\T\setminus I_\varepsilon)}\frac{|dz|}{1-|z|^2}\le
C \int_{\varepsilon}^{\pi} \frac{h_U(t)}{U(t)} dt\]
for a positive constant $C$.
\end{proof}

We also need a suitable curve $\Gamma$ so that Lemma~\ref{discretize} applies. To this end,  we introduce the curve $\psi_U(e^{it}):=\exp(-U(t)/h_U(t)+it)$.

\begin{lemma}\label{poincare}
Suppose that $U$ is in $\Us^*$.
\begin{itemize}
\item[(a)] There exist constants $c$ and $C$ such that 
\[ ch_U(t) \le |\varphi_U'(\psi_U(t))| |\psi_U'(t)| \le C h_U(t) \]
whenever $t$ is sufficiently small. 
\item[(b)] Arclength  on $\varphi_U\circ\psi_U(\T)$ is a Carleson measure for $\Ht$.
\item[(c)] There exists a positive constant $c$ such  that
\[ \ell_P(\varphi_U\circ \psi_U(\T\setminus I_\varepsilon))\ge c\int_{\varepsilon}^\pi \frac{h_U(t)}{U(t)} dt \]
whenever $\varepsilon> 0$.
\end{itemize}
\end{lemma}
\begin{proof}
We simplify the notation by setting $\varphi=\varphi_U$ and $\psi=\psi_U$. We begin by writing 
\[ z(t)=\rho(t)e^{i \tau(t)}:=\varphi(\psi(t))\]
and compute derivatives:
\begin{eqnarray}\label{tau} \tau'(t)&=&v'_r(\psi(t))|\psi(t)|(U(t)/h_U(t))'-v'_\theta(\psi(t)), \\
\label{rho} \rho'(t)&=&\rho(t)(u'_r(\psi(t))|\psi(t)|(U(t)/h_U(t))'-u'_\theta(\psi(t)). \end{eqnarray}
Since $U$ is in $\Us^*$, we have
\[ \left(\frac{U(t)}{h_U(t)}\right)'=\frac{U'(t)}{h_U(t)}+\left(\frac{U(t)}{t h_U(t}\right)^2 =o(1) \]
when $t\to 0$. Hence, by \eqref{hv} and \eqref{ut} of Lemma~\ref{conjugates} and the Cauchy--Riemann equations in polar coordinates,
\eqref{tau} and \eqref{rho} lead respectively to the estimates
\begin{eqnarray}\label{tau2}ch_U(t) \ \, \le \, \ \tau'(t)\ &\le & C h_U(t), \\
\label{rho2}| \rho'(t)| &\le & C' h_U(t) \end{eqnarray}
for sufficiently small $t$. 

Part (a) is immediate from \eqref{tau2} and \eqref{rho2}. Since \eqref{tau2} and \eqref{rho2} also imply that
$| \rho'(t)|\le C \tau'(t)$, it follows that 
\[ \int_{z(t)\in Q} |dz| \le C \ell(Q) \]
when $Q$ is a Carleson square,
and therefore arclength on $\varphi\circ\psi(\T)$ is a Carleson measure. Finally, since  $1-|z(t)|^2\le C U(t)$, we can use 
the left inequality in \eqref{tau2} to infer that
\[ \ell_P(\varphi\circ\psi(\T\setminus I_\varepsilon))=2\int_{\varphi(\T\setminus I_\varepsilon)}\frac{|dz|}{1-|z|^2}\ge
c \int_{\varepsilon}^{\pi} \frac{h_U(t)}{U(t)} dt\]
for some constant $c$.
\end{proof}

\subsection{Proofs of the bounds from above}

The following lemma yields the desired estimate for $\|\mu_{\varphi_U,r}\|_{\mathcal C}$. 

\begin{lemma}\label{Carlesonsquare2}
Suppose that $U$ belongs to $\Ug$ and that there exists a constant $c$ such that $ch_U(t)\le |V'(t)|$. If $Q$ is a Carleson square in $\D$ of side length 
$\delta \le U(\varepsilon)$, then
 \[ \mu_{\varphi_U,r}(Q)\le (C/h_U(\varepsilon)) \delta\]
 when $\varepsilon\to 0$, where the constant $C$ depends on $U$ but not on $Q$ or $r$. 
 \end{lemma}

\begin{proof} We observe that it suffices to estimate the normalized Lebesgue measure of the set
\[ A_{\delta}(t):=\left\{\tau: \ |V(t)-V(t+\tau)|\le \delta, \  1-\exp(-U(\tau))\le \delta\right\} \]
for every fixed $t$ such that $1-\exp(-U(t))\le \delta$.  If $\tau$ belongs to $A_{\delta}(t)$, then $1-\exp(-U(\xi))\le \delta$ for every $\xi$ between $t$ and $\tau$. Such $\xi$ will satisfy the inequality $\xi\le 2\varepsilon$ if $\varepsilon$ is small enough. By the mean value theorem and the assumption on $U$,
\[ |V(t)-V(t+\tau)| = |V'(\xi) | |\tau| \ge c h_U(\xi) |\tau| \ge c'h_U(2\varepsilon) |\tau|, \]
which gives the desired result.
 \end{proof} 

In view of the left inequality in \eqref{Vest}, this lemma applies when $U$ is in $\Us^*$. Plainly, it can also be used when $U$ is as in Theorem~\ref{sousex}. Although the estimate is independent of $r$, the main point is that $\|\mu_{\varphi_U,r}\|_{\mathcal C}\le C /h_U(\varepsilon)$ when
$r=\exp(-U(\xi))$.
 
\begin{proof}[Proof of the bound from above in part (a) of Theorem~\ref{smoothcusp}] We choose $B$ as in Subsection~\ref{First}. Then Lemma~\ref{variant1} and Lemma~\ref{poincare1} give that
\[ \sup_{z\in \T\setminus I_{\varepsilon}} |B(\varphi_{U}(z))|\le \exp\left[-c n \left(\int_{\varepsilon}^{1}\frac{h_U(t)}{U(t)}dt\right)^{-1}\right] \]
for some constant $c$ if, say, $\varepsilon<1/2$. Using also Lemma~\ref{Carlesonsquare2} and part (a) of Theorem~\ref{abovebelow}, we therefore obtain the bound
\[ a_n(C(\varphi_U))\ll \max\left\{[h_U(\varepsilon)]^{-1/2},\exp\left[-\frac{c n}{2} \left(\int_{\varepsilon}^{1}\frac{h_U(t)}{U(t)}dt\right)^{-1}\right]\right\}. \]
If we now choose $\varepsilon=\gamma_U^{-1}(cn)$, then
\[ cn=\int_{\varepsilon}^{1}\frac{h_U(t)}{U(t)}dt \times (\log h_U(\varepsilon)-\log h_U(1)) \]
by the definition of $\gamma_U$,
and the result follows.
\end{proof}

\begin{proof}[Proof of the bound from above in part (a) of Theorem~\ref{extreme}]
In view of the preceding proof, it is sufficient to show that 
\begin{equation}\label{reduce} h_U(\gamma_U^{-1}(Cn)) \gg h_U(e^{-\sqrt{n}}) \end{equation}
when $U$ satisfies the hypotheses in part (a) of Theorem~\ref{extreme}. Our assumption on $U'/U$ implies that 
\[ \log U(1/2)-\log U(t)\le - \log(2t)-c(\log|\log 2|-\log|\log t|),\]
while our hypothesis on $U/h_U$ gives that $\log h_U(t)\le C \log \log |\log t|$ for, say, $t\le e^{-2}$. Inserting these estimates into the definition of $\gamma_U$, we get that
\[ \gamma_U(t)\le |\log t|^c\]
for some $c>2$ and sufficiently small $t$. This may be rephrased as the statement that $\gamma_U^{-1}(x)\le \exp(-x^{1/c})$ for sufficiently large $x$. Since
\[ \frac{h_U(e^{-\sqrt{n})}}{h_U(e^{-(Cn)^{1/c}})}\ll 1\]
by our boundedness condition on  $|h'_U(t)|/h_U(t)|$, we arrive at \eqref{reduce}.
\end{proof}

\begin{proof}[Proof of the bound from above in Theorem~\ref{sousex}]
The proof follows the same pattern as above. 
Since $\varphi_U(e^{it})=(1+(1-e^{it})^{\alpha})^{-1}$, an explicit computation gives
\[ \ell_P(\varphi_U)(\T\setminus I_{\varepsilon})=\frac{\alpha}{\cos(\alpha \pi/2)}(1+o(1)) |\log \varepsilon| \]
when $\varepsilon\to 0$. Hence, by Lemma~\ref{variant1}, we have
\[ \sup_{|t|>\varepsilon} |B(\varphi(e^{it})|\le \exp(-(1-\alpha)\pi^2 n/(2\alpha))\,|\log \varepsilon|\]
for sufficiently small $\varepsilon$. We use again Lemma~\ref{Carlesonsquare2} and part (a) of Theorem~\ref{abovebelow}; since $h_U(\varepsilon)$ behaves as $\varepsilon^{\alpha-1}$, we finish the proof by choosing $\varepsilon$ such that 
\[|\log \varepsilon|=\pi\sqrt{n/(2\alpha)}.\]
\end{proof}
\subsection{Proof of the bounds from below} 
In the two first of the following proofs, we set $\varepsilon=e^{-\sqrt{n}}$. 

\begin{proof}[Proof of the bound from below in part (a) of Theorem~\ref{smoothcusp}]
We set $Z=(z_j)$, where
\[ z_j:=\psi_U(e^{i\theta_j}) =\exp(-U(\theta_j)/h_U(\theta_j)+i \theta_j)\] 
is determined by the requirement that $\varphi_U(Z)$ be a sequence of points on $\Gamma:=\varphi_U\circ\psi_U(\T\setminus I_\varepsilon)$ as constructed in the proof of Lemma~\ref{discretize}. This implies that we have 
\[ M(\varphi_U(Z))\le  \exp(c n/\ell_P(\varphi_{U}\circ\psi_{U}(\T\setminus I_\varepsilon)))\] for some constant $c$.
Since 
\[ |\theta_j-\theta_{j'}|=\left|\int_{\theta_j}^{\theta_{j'}} dt\right|= \int_{\theta_j}^{\theta_{j'}} \frac{|dz(t)|}{|\varphi_U'(\psi_U(t))| |\psi_U'(t)|} \]
the right inequality in part (a) of Lemma~\ref{poincare} gives
\[ |\theta_j-\theta_{j'}| \ge \frac{c}{h_U(\theta_j)} \int_{\varphi_U(z_j)}^{\varphi_U(z_{j'})}ds,\]
where 
the integration is along the curve $\Gamma$ and we have assumed that $\theta_{j} <\theta_{j'}$. Since $d(\varphi_U(z_j),\varphi_U(z_{j'})\ge c\ell_P(\Gamma)/n$ for two distinct points $z_j$ and $z_{j'}$ from $Z$, we therefore get
\[ |\theta_j-\theta_{j'}| \ge \frac{c u(z_j)}{h_U(\theta_j)} \ell_P(\Gamma)/n \ge c'(1-|z_j|) \ell_P(\Gamma)/n,\]
where we in the last step used \eqref{uh}. It follows from this that
\[ \upsilon_{Z}(Q)/\ell(Q)\le c\,n/\ell_P(\varphi_U\circ\psi_U(\T\setminus I_\varepsilon))\]
for every Carleson square $Q$.
By part (b) of Theorem~\ref{abovebelow}, we therefore have
\[ a_n(\Cp) \ge c  \exp(-cn/\ell_P(\varphi_U\circ\psi_U(\T\setminus I_\varepsilon))-1/2 \log h_U(\varepsilon)),\]
where we also used that $(1-|z_j|^2)/(1-|\varphi_U(z_j)|^2)\ge c h_U(\theta_j)$ holds in view of \eqref{uh}.
Since 
\[ \ell_P(\varphi_U\circ\psi_U(\T\setminus I_\varepsilon))\ge c \gamma_U(\varepsilon)/(\log h_U(\varepsilon)-\log h_U(1))
\] for small $\varepsilon$, the 
desired estimate follows. 
\end{proof}

\begin{proof}[Proof of the bound from below in part (a) of Theorem~\ref{extreme}] In view of the bound from below in part (a) of Theorem~\ref{smoothcusp}, we only need to prove that \eqref{expaway} holds, but as already noted, this is a consequence of our assumption on $U(t)/(th_U(t))$. 
\end{proof}

\begin{proof}[Proof of the bound from below in Theorem~\ref{sousex}]
We start with a computation in the upper half-plane, where the pseudohyperbolic distance is defined as
\[ \varrho(z,w):=\left|\frac{z-w}{z-\overline{w}}\right|.\] 
Set $W=(e^{i\theta} \lambda^j)_{j\in \Z}$, where $0<\theta <\pi/2$ and $0<\lambda=:e^{-\varepsilon}<1$ are parameters to be chosen later. A computation shows that
\[ \delta(W)= \exp\left[2 \sum_{j=1}^{\infty} (\log(1-\lambda^j)-\log|1-e^{i2\theta }\lambda^j|)\right]=:\exp[S]. \]
By expanding the logarithms as power series and permuting sums, we get
\[S= -2 \sum_{n=1}^{\infty}\frac{ (1-\cos 2n\theta )}{n}\frac{\lambda^n}{1-\lambda^n}=-2\sum_{n=1}^{\infty}\frac{ (1-\cos 2n\theta )}{n^{2}(1-\lambda)}+o\left(\frac{1}{1-\lambda}\right),\] 
which implies that
\[ \delta(W)= \exp\left[-(2+o(1))\varepsilon^{-1} \sum_{n=1}^{\infty}\frac{ (1-\cos 2n\theta )}{n^{2}}\right]\]
when $\varepsilon\to 0$. Since the function $f(x)=(x^2-\pi^2/3)/4$ has Fourier series expansion $\sum_{n=1}^\infty (-1)^n n^{-2}\cos nx$ on $[-\pi,\pi]$, this gives
\[ \delta(W) =\exp\left[-(2+o(1))\varepsilon^{-1} (f(\pi)-f((\pi-2\theta))\right],\]
or, in other words, by the definition of $f$,
\begin{equation}\label{deltacomp} \delta(W) = \exp\left[-2(\pi \theta-\theta^2+ o(1))\varepsilon^{-1} \right] \end{equation}
when $\varepsilon\to 0$.

We now assume that $\theta>(1-\alpha)\pi/2=:\theta_0$. This means that for sufficiently large $j$, we may choose $z_j$ such that 
\begin{equation}\label{varphiz} \varphi(z_j)=\frac{i-\lambda^je^{i\theta}}{i+\lambda^j e^{i\theta}}=\frac{1-\lambda^je^{i(\theta-\pi/2)}}{1+\lambda^j e^{i(\theta-\pi/2)}}. \end{equation}
Indeed, inverting the explicit expression $\varphi(z)=(1+(1-z)^{\alpha})^{-1}$, we get
\begin{equation}\label{asymp}
z_j=1-2^{1/\alpha}\lambda^{j/\alpha} e^{i(\theta-\pi/2)/\alpha} (1+O(\lambda^j))
\end{equation}
when $j\to \infty$, and hence
\begin{equation}\label{ratio}
 \frac{1-|z_j|^2}{1-|\varphi(z_j)|^2}\ge c \lambda^{(1-\alpha)j/\alpha} 
\end{equation}
for sufficiently large $j$, where $c$ depends on $\theta$. We fix $j_0$ such that \eqref{ratio} holds for all $j> j_0$ and set $Z=(z_j)_{j=j_0+1}^{j_0+n}$. Note that $\delta(\varphi(Z))\ge \delta(W)$ by conformal invariance. From the explicit expressions \eqref{varphiz} and \eqref{asymp}, we get that both $\|\upsilon_{\varphi(Z)}\|_{\mathcal C}$ and $\|\upsilon_Z\|_{\mathcal C}$ are bounded by a constant times $(1-\lambda)^{-1}$
and in particular
\[ M(\varphi(Z))\le \exp\left[-2(\pi \theta-\theta^2+ o(1))\varepsilon^{-1} \right] \]
in view of \eqref{deltacomp} and Theorem~\ref{CSS}. Therefore, by part (b) of Theorem~\ref{abovebelow} and \eqref{ratio}, we obtain 
\begin{equation}\label{Mest} a_n(C_{\varphi}) \ge \exp\left[-2(\pi \theta-\theta^2+ o(1))\varepsilon^{-1})-(1-\alpha+O(\varepsilon))n\varepsilon/(2\alpha)\right] \end{equation}
when $\varepsilon\to 0$. We now choose $\theta$ so  that $\pi \theta -\theta^2<(1-\alpha)\pi^2/2$. This is compatible with the previous requirement on $\theta$ since $$\pi\theta_0-\theta_{0}^{2}=\big(1-\alpha\big)\big(\frac{1+\alpha}{2}\big)\pi^{2}/2<(1-\alpha)\pi^2/2.$$  Then, in view of \eqref{Mest}, we get
$$a_{n}(C_\varphi)\ge c\exp\Big[-(1-\alpha)(\pi^{2}\varepsilon^{-1}+\varepsilon n/(2\alpha) )\Big].$$ The optimal choice 
$\varepsilon=\pi \sqrt{2\alpha/n}$ gives the desired bound from below in Theorem~\ref{sousex}.  \end{proof}

\section{Proofs when $U$ is in $\Uc$}\label{Ucusp}

\subsection{Estimates for conjugate functions}
In our proof of the bound from below in part (b) of Theorem~\ref{extreme}, we will need the following estimate.
\begin{lemma}\label{conjugatec} Suppose that $U$ belongs to $\Uc$. 
\begin{itemize}
\item[(a)] If $U'(t)/U(t)\leq\alpha/t$ for $0<\alpha<1$, then
\[|V(t)|/U(t)\le C\sqrt{\alpha}/ (1-\alpha) \]
for some constant $C$.
\item[(b)]
If $U'(t)/U(t)=o(1/t)$ when $t\to 0$, then
\[ V(t)/U(t)\to 0\] when $t\to 0$.
\end{itemize}
\end{lemma}

\begin{proof} We start from the formula
\[ V(t)=\int_{-\pi}^{\pi} U(x)\cot\left(\frac{t-x}{2}\right) \frac{dx}{2\pi}, \]
where the integral is to be understood in the principal value sense. We may write this as
\[ V(t) =   \int_{0}^{t} (U(t-x)-U(t+x)) \cot\left(\frac{x}{2}\right) \frac{dx}{2\pi} - 
\int_{t}^{\pi} (U(x+t)-U(x-t)) \cot\left(\frac{x}{2}\right) \frac{dx}{2\pi}. \]
We have
\[  \left|\int_{0}^{t} (U(t-x)-U(t+x)) \cot\left(\frac{x}{2}\right) \frac{dx}{2\pi}\right|\le C \left (t \max_{\varepsilon\le x\le 2t}  U'(x) +
\frac{\varepsilon U(2t)}{t}\right)\]
and
\[ \left|\int_{t}^{\pi} (U(x+t)-U(x-t)) \cot\left(\frac{x}{2}\right) \frac{dx}{2\pi}\right|\le C \left (\frac{\varepsilon U(2t+\varepsilon)}{t}+
 t \int_{t+\varepsilon}^{\pi}\max_{|x-\xi|\le t}U'(\xi) \frac{dx}{x} \right),\]
 where in both cases $0<\varepsilon\le t/2$. The result now follows if we use the respective assumptions on $U'/U$; in part (a), we choose $\varepsilon=\min(t/2,t \sqrt{\alpha})$, while in part (b), we may choose $\varepsilon$ arbitrarily small if $t$ is sufficiently small. 
\end{proof}

\subsection{Proof of the bounds from above}

We have two cases to consider: the bound from above in respectively parts (b) and (a) of  Theorem~\ref{smoothcusp} and Theorem~\ref{extreme}. We will apply part (a) of Theorem~\ref{abovebelow} and in each case choose $r=1-\exp(-\tau \omega_U(n))$ for a suitable constant $\tau$ and $B$ as in Lemma~\ref{variant2}.  Since 
\[\ell_P([0,r])=(1+o(1))|\log(1-r)|=(\tau+o(1))\omega_U(n),\] 
we will have
\[ \sup_{z\in \Omega_r} |B(\varphi(z))| \le \exp\left(-(\beta+o(1))\,   n/(\tau \omega_U(n))\right),\]
where $\beta$ is as in \eqref{alpha}. We observe that part (a) of Lemma~\ref{conjugatec} gives the desired estimate in part (b) of Theorem~\ref{smoothcusp}. If we choose $\tau=1+o(1)$ when $n\to \infty$, then part (b) of Lemma~\ref{conjugatec} gives $\beta=\pi^2/2+o(1)$ because then $\lambda\to 1$.  

It remains to determine the contribution from the Carleson norm on the right-hand side of the bound in part (b) of Theorem~\ref{smoothcusp}.
The following lemma yields the required estimates.

\begin{lemma}\label{Carlesonsquare}
Suppose that $U$ belongs to $\Uc$ and that $\tau>1$. 
\begin{itemize}
\item[(a)] If $U'(t)/U(t)\le (1-\delta)/t$ and $n$ is sufficiently large, then for every Carleson square
$Q$  in $\D$ of side length $\varepsilon\le 
e^{- \tau\omega_U( n)}$, we have
 \[ \mu_{\varphi_U,r}(Q)\le e^{-\delta n/\omega_U(n)} \varepsilon. \]
\item[(b)] If $\eta_U'(x)/\eta_U(x)=o(1/x)$ when $x\to\infty$ and $n$ is sufficiently large, then for every Carleson square
$Q$ in $\D$ of side length $\varepsilon\le 
e^{- \tau\omega_U( n)}$, we have
 \[ \mu_{\varphi_U,r}(Q)\le e^{- 8 n/\omega_U(n)} \varepsilon. \]
\end{itemize}
  \end{lemma}
Here the constant $8$ in the exponent in part (b) is somewhat arbitrary; the point is just that we have a constant that is larger than the sharp constant $\pi^2/2$ appearing on the left-hand side of (b) in Theorem~\ref{smoothcusp}.

\begin{proof}[Proof of Lemma~\ref{Carlesonsquare}] The first part of the proof is the same in both cases. We need to estimate the normalized Lebesgue measure of the set
\[ A_{\varepsilon}:=\{t: \ |\varphi_{U}(e^{it})|\ge 1-\varepsilon \}. \]
Since $|\varphi_U|=e^{-u}$, we have $A_{\varepsilon}\subset \{t: \ U(t) \le \varepsilon +\xi \varepsilon^2\}$ 
for some constant $\xi$. Hence
\[ \sigma(A_{\varepsilon})\le \pi^{-1} e^{-\eta_{U}^{-1}(|\log(\varepsilon+\xi \varepsilon^2)|)}\]
by the definition of the function $\eta$. Writing
\[ e^{-\eta_U^{-1}(|\log(\varepsilon+\xi \varepsilon^2)|)}
=e^{-(1-\delta)\eta_U^{-1}(|\log(\varepsilon+\xi \varepsilon^2)|)}\times e^{-\delta \eta_U^{-1}(|\log(\varepsilon+\xi \varepsilon^2)|)} \]
and using that $\eta_U^{-1}(x)\ge (1-\delta)^{-1} x $, we get
\[ \sigma(A_{\varepsilon})\le \pi^{-1}(\varepsilon+\xi \varepsilon^2) e^{-\delta\eta_{U}^{-1}(|\log(\varepsilon+\xi \varepsilon^2)|)}.\]
Taking into account that  $\varepsilon+\xi \varepsilon^2 \le  e^{-\omega_U( n)}$ for sufficiently large $n$, the proof of part (a) is complete if we in the final step use the definition of $\omega_U$. 

As to part (b), we note we have $\varepsilon+\xi \varepsilon^2 \le  e^{-\tau' \omega_U( n)}$ for $1<\tau'<\tau$ and sufficiently large $n$. This means that we have
\[ \sigma(A_{\varepsilon})\le \pi^{-1}(\varepsilon+\xi \varepsilon^2) e^{-\delta\eta_U^{-1}(\tau' \omega_U(n))}.\]
We observe that the condition on $\eta_U$ implies that
\[ \eta_U^{-1}(x)=o(\eta_U^{-1}(\tau' x)) \]
when $x\to\infty$ for every $\tau'>1$, and use again in the final step the definition of $\omega_U$.
\end{proof}

In view of Theorem~\ref{abovebelow}, part (a) finishes the proof of the estimate from above in Theorem~\ref{smoothcusp} since we may choose 
any $\tau>1$ and thus
$\kappa(\delta)=\min(\delta, \beta-\varepsilon)$ for an arbitrary $\varepsilon>0$. Here $\kappa(\delta)=\delta$ for $\delta$ close to $1$ because  
$\beta\to \pi^2/4$ when $\delta\to 1$ by part (a) of Lemma~\ref{conjugatec}. Part (b) of Lemma~\ref{Carlesonsquare}, on the other hand, justifies the choice $\tau=1+o(1)$ when $U$ satisfies the hypothesis in part (b) of Theorem~\ref{extreme}, and we obtain therefore the desired
estimate $a_n(C_{\varphi_U})\ll \exp(-(\pi^2/2+o(1))n/\omega_U(n))$ when $n\to\infty$.

\subsection{Proof of the bounds from below}
The main estimate required in this case, is contained in the following lemma. Here 
\[ Z(\lambda):=\left\{ \frac{1-\lambda^j}{1+\lambda^j}: j=1,2,\ldots\right\} \]
for $0<\lambda<1$.

\begin{lemma}\label{intconst}
We have
\[ \delta(Z(\lambda))\ge e^{-(\pi^2/2)/(1-\lambda)}. \]
\end{lemma}
\begin{proof}
Since 
\[ \rho\left(\frac{1-\lambda^j}{1+\lambda^j},\frac{1-\lambda^k}{1+\lambda^k}\right)=\left\vert\frac{\lambda^k-\lambda^j}{\lambda^k+\lambda^j}\right\vert,\] we find first that
\begin{equation}\label{delta} \delta(Z(\lambda))\ge \prod_{j=1}^{\infty} \left(\frac{1-\lambda^{j}}{1+\lambda^j}\right)^2
=\exp\Big(-2\sum_{j=1}^{\infty}\big(\log(1+\lambda^j)-\log(1-\lambda^j)\big)\Big) .\end{equation}
We have 
\begin{eqnarray*} 2\sum_{j=1}^{\infty}\big(\log(1+\lambda^j)-\log(1-\lambda^j)\big)& = & 4\sum_{n=0}^\infty \sum_{j=1}^{\infty} \frac{\lambda^{j(2n+1)}}{2n+1}  = 4 \sum_{n=1}^{\infty} \frac{\lambda^{2n+1}}{(2n+1)(1-\lambda^{2n+1})}\\ & \le & \frac{4}{(1-\lambda)} \sum_{n=0}^{\infty}\frac{1}{(2n+1)^2}=
\frac{\pi^2}{2 (1-\lambda)}, \end{eqnarray*}
where we used that $ 1-\lambda^{2n+1}\ge (2n+1)(1-\lambda) \lambda^{2n+1}$. Returning to \eqref{delta}, we arrive at the desired estimate.
\end{proof}

It should be noted that the computation in the preceding proof  was also used in \cite{LIQUEROD}, Lemma 6.4, and in the present paper, under a more elaborate form, for getting the lower bound in Theorem \ref{sousex}.

Since 
\[  \upsilon_{Z(\lambda)}(Q)/\ell(Q)\le 2 (1-\lambda)^{-1},\]
Lemma~\ref{intconst} gives immediately
\begin{equation}\label{constinter}
M(Z(\lambda))\le \exp((\pi^2/2+o(1))/(1-\lambda))
\end{equation}
when $\lambda\to 1$.

Before turning to the proof of the remaining bounds from below, we establish the following estimate.
 
\begin{lemma}\label{omega}
Suppose that $U$ is in $\Uc$ and that 
$\eta_U'(x)/\eta_U(x)=o(1/x)$ when $x\to \infty$.  Then
\[\omega_U'(x)/\omega_U(x)\le o(1/x)\]
when $x\to \infty$.
\end{lemma} 
\begin{proof}
We compute the derivative
\[\label{omegadif}\omega_U'(x)=\eta_U'\left(\frac{x}{\omega_U(x)}\right)\left[\frac{1}{\omega_U(x)}-\frac{\omega'_U(x)x}{[\omega_U(x)]^2}\right]\le
\eta_U'\left(\frac{x}{\omega_U(x)}\right)/\omega_U(x). \]
We obtain the desired result from this formula by using the defining identity $\omega_U(x)=\eta(x/\omega_U(x))$. 
\end{proof}
 
\begin{proof}[Proof of the bound from below in part (b) of Theorem~\ref{smoothcusp}] 
We set \[z_j=\varphi^{-1}((1-\lambda^j)/(1+\lambda^j))\] and choose $Z=(z_j)_{j=j_0+1}^{j_0+n}$, where $j_0$ is the smallest positive integer such that we have $(1-\lambda^j)/(1+\lambda^{j})$  in $\varphi_U(\D)$ for every $j>j_0$.  Since
\[ \log \frac{U(t)}{U(\tau)}\le \alpha \log \frac{t}{\tau} \]
when $t>\tau$, we get
\[ \upsilon_{(z_j)}(Q)/\ell(Q)\le 2 (1-\lambda)^{-1}\]
so that 
\begin{equation}\label{carest}
\| \upsilon_{(z_j)} \|_{\mathcal C}\le C (1-\lambda)^{-1}.
\end{equation}
We now choose $\lambda$ such that
\begin{equation}\label{lambda} \frac{1-\lambda^{j_0+n}}{1+\lambda^{j_0+n}}= \varphi(1-e^{-\nu n/\omega_U(\nu n)}), \end{equation}
where $\nu$ is a parameter to be determined below. Then it follows from the definition of $z_n$ that
\begin{equation}\label{onef} 1-z_{n+j_0} = e^{-\nu n/\omega_U(\nu n)}.\end{equation}
Moreover, since
\[ \varphi(1-x)=1-e^{O(1)} U(x) \]
when $x\to 0$,
we may use \eqref{lambda} and the definition of $\omega$ to infer that 
\[ \lambda^{n+j_0} = e^{-\omega_U(\nu n)+O(1)}. \]
Thus we get $|\log\lambda|=\omega_U(\nu n)/n+O(n^{-1})$ and deduce from part (b) of Theorem~\ref{abovebelow} and our three estimates
\eqref{constinter},  \eqref{carest}, and \eqref{onef}  that
\begin{equation}\label{estdelta} a_n(C_{\varphi_U})\ge \exp[-(\pi^2/2+\nu/2+o(1))n/\omega_U(\nu n)]. \end{equation}
The proof is complete if we choose $\nu=1$.
\end{proof}

\begin{proof}[Proof of the bound from below in part (b) of Theorem~\ref{extreme}] The proof is identical to the preceding one up to \eqref{estdelta}. In view of Lemma~\ref{omega}, we may now in the final step choose $\nu=o(1)$ when $\varepsilon\to 0$.
\end{proof}

\section*{Acknowledgement} The authors are grateful to the anonymous referee for a careful review leading to a clarification of several technical details.

\end{document}